\documentclass[leqno]{article}

\usepackage[margin=1.9cm]{geometry}

\usepackage{amsmath,amssymb,amsthm,authblk,MnSymbol,tikz}

\usetikzlibrary{arrows.meta,calc}

\newcommand{\fordsix}[8]{
\begin{tikzpicture}[baseline=(a.base),scale=#1,every node/.style={scale=#2},inner sep=0pt,outer sep=0pt]
\node (a) at (1,0) {$#3$};
\node at (2,0) {$#4$};
\node at (3,0) {$#5$};
\node at (4,0) {$#6$};
\node at (5,.85) {$#7$};
\node at (5,-.85) {$#8$};
\end{tikzpicture}
}

\newlength{\halogap}

\tikzset{Rightarrow/.style={double equal sign distance,>={Implies},->},
triple/.style={-,preaction={draw,Rightarrow}}}

\tikzset{DynkinNode/.style={circle,draw,thick,minimum size=.7em,inner sep=0pt,font=\scriptsize}}

\newcommand{\GDynkin}[1]{
\begin{tikzpicture}[scale=.8,anchor=base,baseline]
\foreach\kthweight[count=\k] in {#1}{
\ifnum\k=1\node[DynkinNode] (1) at (-1,0) {$\scriptscriptstyle\kthweight$};\fi
\ifnum\k=2\node[DynkinNode] (2) at (-2,0) {$\scriptscriptstyle\kthweight$};\fi
}
\draw[->,triple,thick] (2) -- (1);
\end{tikzpicture}
}

\newcommand{\FDynkins}[1]{
\begin{tikzpicture}[scale=.6,anchor=base,baseline]
\foreach\kthweight[count=\k] in {#1}{
\ifnum\k=1\node[DynkinNode] (1) at (-1,0) {$\scriptscriptstyle\kthweight$};\fi
\ifnum\k=2\node[DynkinNode] (4) at (-4,0) {$\scriptscriptstyle\kthweight$};\fi
\ifnum\k=3\node[DynkinNode] (2) at (-2,0) {$\scriptscriptstyle\kthweight$};\fi
\ifnum\k=4\node[DynkinNode] (3) at (-3,0) {$\scriptscriptstyle\kthweight$};\fi
}
\draw[thick] (1) -- (2);
\draw[double equal sign distance,>={Implies},->,thick] (3) -- (2);
\draw[thick] (3) -- (4);
\end{tikzpicture}
}

\newcommand{\FDynkino}[2]{
\begin{tikzpicture}[scale=.6,anchor=base,baseline]
\foreach\kthweight[count=\k] in {#1}{
\ifnum\k=1\node[DynkinNode] (1) at (-1,0) {$\scriptscriptstyle\kthweight$};\fi
\ifnum\k=2\node[DynkinNode] (4) at (-4,0) {$\scriptscriptstyle\kthweight$};\fi
\ifnum\k=3\node[DynkinNode] (2) at (-2,0) {$\scriptscriptstyle\kthweight$};\fi
\ifnum\k=4\node[DynkinNode] (3) at (-3,0) {$\scriptscriptstyle\kthweight$};\fi
}
\draw[thick] (1) -- (2);
\draw[double equal sign distance,>={Implies},->,thick] (3) -- (2);
\draw[thick] (3) -- (4);
\foreach\Node[count=\i] in {#2}
{\coordinate (AuxNode-\i) at ($(\Node)$);}
\draw[red,rounded corners] ($(AuxNode-1)+(-\halogap,\halogap)$) --  ($(AuxNode-2)+(\halogap,\halogap)$)
    -- ($(AuxNode-2)+(\halogap,-\halogap)$) -- ($(AuxNode-1)+(-\halogap,-\halogap)$)--
    cycle;
\end{tikzpicture}
}

\newcommand{\EDynkin}[2]{
\begin{tikzpicture}[scale=.4,anchor=base,baseline]
\foreach\kthweight[count=\k] in {#1}{
\pgfmathtruncatemacro{\prevnode}{\k-1}
\ifnum\k=1\node[DynkinNode] (\k) at (0,1) {$\scriptscriptstyle\kthweight$};\fi % exchanged 1 and 2
\ifnum\k>1\node at (4-\k,0) {$\scriptscriptstyle\kthweight$};
\node[DynkinNode] (\k) at (4-\k,0) {$\scriptscriptstyle\kthweight$};
\ifnum\k>2\draw[thick] (\k) -- (\prevnode);\fi
\ifnum\k=4\draw[thick] (\k) -- (1);\fi
\fi
}
\foreach\Node[count=\i] in {#2}
{
\xdef\imax{\i}
\coordinate (AuxNode-\i) at ($(\Node)$);
}
\ifnum\imax=3%
\draw[red,rounded corners] ($(AuxNode-1)+(-\halogap,\halogap)$) -|  ($(AuxNode-2)+(-\halogap,\halogap)$)
    -- ($(AuxNode-2)+(\halogap,\halogap)$) |-
    ($(AuxNode-3)+(\halogap,\halogap)$)
    --($(AuxNode-3)+(\halogap,-\halogap)$) -- ($(AuxNode-1)+(-\halogap,-\halogap)$) -- cycle;
\else\ifnum\imax=2%
\draw[red,rounded corners] ($(AuxNode-1)+(-\halogap,\halogap)$) --  ($(AuxNode-2)+(\halogap,\halogap)$)
    -- ($(AuxNode-2)+(\halogap,-\halogap)$) -- ($(AuxNode-1)+(-\halogap,-\halogap)$)--
    cycle;
\fi
\fi
\end{tikzpicture}
}

\newtheorem{theorem}{Theorem}[section]
\newtheorem{lemma}[theorem]{Lemma}
\newtheorem{corollary}[theorem]{Corollary}
\newtheorem{proposition}[theorem]{Proposition}

\newtheorem*{lemma*}{Lemma}

\theoremstyle{definition}
\newtheorem{definition}[theorem]{Definition}

\theoremstyle{remark}
\newtheorem{remark}[theorem]{Remark}
\newtheorem{example}[theorem]{Example}
\newtheorem{examples}[theorem]{Examples}

\let\ge\geqslant
\let\le\leqslant
\let\eps\varepsilon

\newcommand\g{{\mathfrak g}}

\DeclareMathOperator\ad{ad}
\DeclareMathOperator\sla{\mathfrak{sl}}

\title{On Dynkin gradings in simple Lie algebras}

\author[$^1$]{A. G. Elashvili}
\author[$^1$]{M. Jibladze}
\author[$^2$]{V. G. Kac}

\affil[$^1$]{Razmadze Mathematical Institute, TSU, Tbilisi 0186, Georgia}
\affil[$^2$]{Department of Mathematics, MIT, Cambridge MA 02139, USA}

\date{}

\begin{document}

\maketitle

\

{\em To Anthony Joseph for his 75th birthday}

\

\

\section*{Introduction}

In this paper we study Dynkin gradings on simple Lie algebras arising from nilpotent elements. Specifically, we investigate abelian subalgebras which are degree 1 homogeneous with respect to these gradings.

The study of gradings associated to nilpotent elements of simple Lie algebras is important since the finite and affine classical and quantum W-algebras are defined using these gradings. In order to study integrable systems associated to these W-algebras, it is useful to have their free field realizations. One of the ways to construct them is to use the generalized Miura map \cite{dSKV, KW}. This construction can be further improved by choosing an abelian subalgebra in the term $\g_1$ of the grading. That is why the description of such subalgebras, especially the ones of dimension equal half of the dimension of $\g_1$ (which is maximal possible), is important.

We show that for each odd nilpotent orbit there always exists a canonically associated ``strictly odd'' nilpotent orbit, which allows us to reduce our investigations to the latter. (Strictly odd means that all Dynkin labels are either 0 or 1.) The rest of the paper is devoted to the investigation of maximal abelian subalgebras in $\g_1$ for strictly odd nilpotents in simple Lie algebras. For algebras of exceptional type we provide tables with largest possible dimensions of such subalgebras in each case. For algebras of classical type, we find expressions for all possible maximal dimensions of abelian subalgebras in $\g_1$, and, based on that, characterize those nilpotents for which there exists such subalgebra of half the dimension of $\g_1$.%changed

\section{Recollections}\label{recoll}

Let us recall the nomenclature for nilpotents in a semisimple Lie algebra $\g$.

Given such a nilpotent $e$, one chooses an $\sla_2$-triple $(e,h,f)$ for it, that is, another nilpotent $f$ such that $[e,f]=h$ is semisimple and the identities $[h,e]=2e$, $[h,f]=-2f$ hold (Jacobson-Morozov theorem; see e.~g. \cite{CM}). The Dynkin grading is the eigenspace decomposition for $\ad h$:
$$
\g=\bigoplus_{j\in\mathbb Z}\g_j.
$$
Then, to $e$ one assigns a combinatorial object which determines it up to isomorphism. It is the \emph{weighted Dynkin diagram} corresponding to $e$, which is the Dynkin diagram of $\g$ with numbers assigned to each node. These numbers are the degrees $\alpha_i(h)$ of simple root vectors $e_i$ with respect to the choice of a Cartan and a Borel subalgebra in such a way that $h$ (resp. $e$) becomes an element of the corresponding Cartan (resp. Borel) subalgebra. The weighted Dynkin diagrams satisfy certain restrictions --- for example, the weights can only be equal to $0$, $1$ or $2$; moreover if $\g$ is simple of type A, then the weights are symmetric with respect to the center of the diagram, while for types B, C or D there is no weight $1$ occurring to the left of $2$.

The nilpotent is called \emph{even} if there are no $1$'s in its weighted Dynkin diagram, \emph{odd} if it is not even, and \emph{strictly odd} if there are no $2$'s.

It is clear that for even nilpotents the question about abelian subspaces in $\g_1$ is trivial since $\g_1$ is zero.

We will also need the following fact from \cite{Elashvili}:
\begin{proposition}\label{g1g0}
The degree $1$ part $\g_1$ of $\g$ with respect to the grading induced by a nilpotent $e\in\g$ is generated as a $\g_0$-module by those simple root vectors of $\g$ which have weight $1$ in the weighted Dynkin diagram corresponding to $e$.
\end{proposition}\qed

If $\g$ is a simple Lie algebra of classical type, one can assign to $e$ another combinatorial object --- a partition $\lambda_n\geqslant\lambda_{n-1}\geqslant\cdots$ which records dimensions of irreducible representations of $\sla_2$ into which the standard representation of $\g$ decomposes as a module over its subalgebra $(e,h,f)$. Alternatively, the partition consists of sizes of Jordan blocks in the Jordan decomposition of $e$ as an operator acting on the standard representation of $\g$. The partitions are restricted in a certain way, according to the type of $\g$. For type A one may have arbitrary partitions. For types B and D, all even parts must have even multiplicity, while for type C all odd parts must have even multiplicity. These conditions are sufficient as well as necessary, that is, any partition satisfying these conditions corresponds to a nilpotent orbit in a simple Lie algebra of the respective classical type.

Let us recall how one switches from a partition representing a nilpotent to its weighted Dynkin diagram (cf. \cite{SS}).

Each $\lambda_k$ in the partition represents a copy of the $\lambda_k$-dimensional irreducible representation of $\sla_2$, with eigenvalues of $h$ equal to
$$
1-\lambda_k,3-\lambda_k,...,\lambda_k-3,\lambda_k-1.
$$
To obtain the weighted Dynkin diagram one collects from each $\lambda_k$ those eigenvalues, arranges them in decreasing order, and takes consecutive differences.

For example, take the partition $8,6,3,3,2,1,1$. This gives the following eigenvalues of $h$:

\

\

$$
\begin{tabular}{rrrrrrrrrrrrrrr}
-7&  &-5&  &-3&  &-1& &1&  &3&  &5&  &7\\
  &  &-5&  &-3&  &-1& &1&  &3&  &5\\
  &  &  &  &  &-2&  &0& &2\\
  &  &  &  &  &-2&  &0& &2\\
  &  &  &  &  &  &-1& &1\\
  &  &  &  &  &  &  &0\\
  &  &  &  &  &  &  &0
\end{tabular}
$$

\

\

Arranging all numbers from this table in the decreasing order gives
$$
\begin{tabular}{cccccccccccccccccccccccc}
\phantom-7&\phantom-5&\phantom-5&\phantom-3&\phantom-3&\phantom-2&\phantom-2&\phantom-1&\phantom-1&\phantom-1&\phantom-0&\phantom-0&\phantom-0&\phantom-0&-1&-1&-1&-2&-2&-3&-3&-5&-5&-7.
\end{tabular}
$$
Taking the consecutive differences then gives
$$
\begin{tabular}{cccccccccccccccccccccccc}
\phantom-2&\phantom-0&\phantom-2&\phantom-0&\phantom-1&\phantom-0&\phantom-1&\phantom-0&\phantom-0&\phantom-1&\phantom-0&\phantom-0&\phantom-0&\phantom-1&\phantom-0&\phantom-0&\phantom-1&\phantom-0&\phantom-1&\phantom-0&\phantom-2&\phantom-0&\phantom-2
\end{tabular}
$$
which is already the weighted Dynkin diagram of the nilpotent in case of type A.

For types B, C, D one has to leave only left half of the obtained sequence (which obviously is centrally symmetric); more precisely, for an algebra of rank $r$, the first $r-1$ nodes of the weighted Dynkin diagram are as stated, while the rightmost node is defined in a specific way, depending on the type. We skip this part, as it will not play any r\^ole for us; details can be found in e.~g. \cite[Section 5.3]{CM}.

For example, the same partition $8,6,3,3,2,1,1$ also encodes a nilpotent orbit in a simple Lie algebra of type C, since all of its odd parts come with even multiplicities. Then, the weighted Dynkin diagram of this nilpotent is
$$
\begin{tabular}{cccccccccccccccccccccccc}
2&\,0&\,2&\,0&\,1&\,0&\,1&\,0&\,0&\,1&\,0&\,0.
\end{tabular}
$$

It is easy to see from the above procedure that the resulting weighted Dynkin diagram begins with certain sequence of $0$'s and $2$'s; if the largest part of the partition is $\lambda_n$ with multiplicity $m_n$, and the parts of the same parity following it are $\lambda_{n-1}$ with multiplicity $m_{n-1}$, $\lambda_{n-2}$ with multiplicity $m_{n-2}$, ..., $\lambda_{n-k+1}$ with multiplicity $m_{n-k+1}$, while the next part $\lambda_{n-k}$ has the opposite parity, then
the first $1$ appears at the $(km_n+(k-1)m_{n-1}+...+2m_{n-k+2}+m_{n-k+1})$-st place. For the type A it reflects symmetrically, thus having weights $2$ and $0$ at both ends and weights $1$ and $0$ in the middle, while for types B, C or D it starts with a sequence of weights $0$ and $2$ followed by a sequence of weights $0$ and $1$, without any further $2$'s.

According to the above procedure for assigning to a partition a weighted Dynkin diagram, it is easy to see the following
\begin{proposition}
A nilpotent in a simple Lie algebra of classical type is even iff all the parts of the corresponding partition are of the same parity, is odd iff there are some parts with different parities, and strictly odd iff the largest part and the next largest part differ by $1$.
\end{proposition}\qed

\section{Important reduction}

Let $V$ and $U$ be finite-dimensional modules over a reductive Lie algebra $\g$ and let $V\otimes V\to U$ be a $\g$-module homomorphism. It is thus a $\g$-equivariant algebra structure on $V$ with values in $U$.

\begin{proposition}\label{basisprop}
Suppose that there exists an abelian subalgebra of dimension $d$ of the algebra $V$. Then there exists an abelian subalgebra of the algebra $V$ of dimension $d$, spanned by weight vectors of $V$.
\end{proposition}

\begin{proof}[Proof \emph{(proposed by the referee)}]
It follows from Borel's fixed point theorem. Indeed, the Cartan subgroup acts on the complete variety of $d$-dimensional abelian subalgebras of $V$, hence has a fixed point.
\end{proof}

Using this, in what follows we will assume throughout that for a simple Lie algebra of classical type we are given a basis in the standard representation consisting of weight vectors corresponding to the weights $\pm\eps_i$, $i=1,...,n$ and moreover, for the type B, to the zero weight. In the adjoint representation, accordingly, we will have a basis corresponding to ${}\pm\eps_i\pm\eps_j$, $i\ne j$ (accounting for tensor products of basis vectors of the standard representation corresponding to $\pm\eps_i$ and to $\pm\eps_j$) and moreover, for the type B only, those corresponding to $\pm\eps_i$ (accounting for tensor product of a basis vector corresponding to $\pm\eps_i$ and that corresponding to the zero weight) and, for C only, corresponding to $\pm2\eps_i$ (accounting for the tensor product of a basis vector of the standard representation corresponding to $\pm\eps_i$ with itself), $i=1,...,n$.
%!
\begin{proposition}\label{reduction}
For any weighted Dynkin diagram corresponding to a nilpotent $e$ in a simple Lie algebra $\g$, consider a subdiagram obtained as a result of erasing all nodes with weight $2$. Consider the resulting subdiagram together with the remaining weights. Then all connected components of this subdiagram, except possibly one of them, have all weights equal to zero. Moreover this one component (if it exists) is a weighted Dynkin diagram of some strictly odd nilpotent orbit in the diagram subalgebra $\tilde\g\subseteq\g$ of the type determined by the shape of the component.
\end{proposition}

\begin{proof}
For algebras of classical type, this is proved in \ref{mainlemma} below. For algebras of type G$_2$ this is clear as all nilpotents in them are either even or strictly odd. As for exceptional Lie algebras of types E or F, the assertion can be seen to be true directly from looking at the tables {\bf F4o, E6o, E7o, E8o} given in the last section.
\end{proof}

\begin{corollary}\label{coreduction}
For any odd nilpotent $e$ in a simple Lie algebra $\g$ there exists a simple diagram subalgebra $\tilde\g\subseteq\g$ and a strictly odd nilpotent $\tilde e\in\tilde\g$ such that
$$
\g_1(e)=\tilde\g_1(\tilde e),
$$
i.~e. the degree 1 homogeneous parts for the grading on $\g$ induced by $e$ and for the grading on $\tilde\g$ induced by $\tilde e$ coincide. In particular, these degree 1 homogeneous parts have the same abelian subspaces.
\end{corollary}
\begin{proof}
Take for $\tilde\g$ the subalgebra corresponding to the connected component of the weighted Dynkin diagram of $e$ as described in \ref{reduction} above. Moreover let $\tilde e$ be any representative from the orbit corresponding to the weights on this connected component --- it exists by \ref{reduction}.

By construction this subalgebra contains all simple root vectors of degree 1, and moreover they will be precisely the root vectors of those simple roots of $\tilde\g$ which contribute to degree $1$ part in the grading induced by $\tilde e$. From \ref{g1g0} we know that $\g_1(e)$ is the $\g_0(e)$-module generated by these root vectors, while $\tilde\g_1(\tilde e)$ is the $\tilde\g_0(\tilde e)$-module generated by them.

Now observe that the only removed nodes which connect with an edge to some node in the remaining connected component have weight $2$, so that all simple root vectors corresponding to removed nodes with weight $0$ commute with every simple root vector in this component.

It follows that the $\g_0(e)$-module generated by the root vectors corresponding to weight 1 nodes is no larger than the $\tilde\g_0(\tilde e)$-module generated by them, i.~e. $\g_1(e)$ coincides with $\tilde\g_1(\tilde e)$.
\end{proof}

\begin{definition}\label{strodef}
For the orbit of an odd nilpotent in a simple Lie algebra $\g$, call its \emph{strictly odd reduction} the nilpotent orbit in the simple Lie algebra $\tilde\g$ obtained as in \ref{coreduction}.
\end{definition}

Given a nilpotent $e\in\g$ as in \ref{reduction}, one can explicitly produce a nilpotent $\tilde e\in\tilde\g$ from the orbit corresponding to its strictly odd reduction in the sense of \ref{strodef} as follows. The nilpotent $e$ clearly lies in the degree 2 subspace $\g_2$ for the corresponding grading. This subspace is a $\g_0$-module and decomposes canonically into the direct sum of its submodule $[\g_1,\g_1]$ and the submodule $\g_2(2)$ generated by the root vectors of $\g$ corresponding to simple roots with weight $2$.

\begin{proposition}\label{e1}
Given a nilpotent $e$, represent it (in a unique way) as a sum $e_1+e_2$ with $e_1\in[\g_1,\g_1]$ and $e_2\in\g_2(2)$. Then the weighted Dynkin diagram of $e_1$ in the subalgebra corresponding to the subdiagram described in \ref{reduction} is given by weights on that subdiagram.
\end{proposition}

\begin{proof}
We have a reductive group $G_0$ corresponding to $\g_0$ acting on $\g_2=[\g_1,\g_1]+\g_2(2)$, with the element $e=e_1+e_2$ having an open orbit in $\g_2$.
This means that $[\g_0,e_1+e_2]=\g_2$. But this implies that $[\g_0,e_1]=[\g_1,\g_1]$ (and similarly for $e_2$).
Hence $G_0e_1$ is an open orbit in $[\g_1,\g_1]$.

Let us consider an intermediate diagram subalgebra $\tilde\g\subseteq\g'\subseteq\g$ corresponding to the (in general disconnected) diagram, obtained by erasing the nodes with weight $2$ but leaving all other nodes together with their weights intact. It is clear from \ref{reduction} that $\g'$ is a direct sum of $\tilde\g$ and some simple algebras of type A. Hence $e_1$, viewed as an element of this direct sum, obviously has zero summands in all these components of type A.

On the other hand from \ref{reduction} we know that there exists a (strictly odd) nilpotent element $\tilde e$ in $[\g_1,\g_1]$, which has the needed Dynkin diagram.
Then just as $e_1$, we can view $\tilde e$ as a nilpotent in $\g'$, having zero summands in all remaining type A components of $\g'$. It is then clear that this nilpotent
will have the weighted Dynkin diagram obtained as in \ref{reduction}. Moreover it will have an open $G_0$-orbit in $[\g_1,\g_1]$, hence it coincides with the $G_0$-orbit of $e_1$, so $\tilde e$ and $e_1$ have the same weighted Dynkin diagram when viewed as nilpotents in $\g'$. Then obviously they will also have the same weighted Dynkin diagram with respect to $\tilde\g$ since the latter is obtained just by throwing out type A components with zero weights only.
\end{proof}

\begin{remark}
It would be convenient to supplement \ref{coreduction} with the explicit construction, from an $\sla_2$-triple $(e,f,h)$ corresponding to a given nilpotent orbit in $\g$, of an $\sla_2$-triple $(\tilde e,\tilde f,\tilde h)$ for its strictly odd reduction as in \ref{strodef}. Since $\tilde\g$ comes with a grading (determined by the weights on the corresponding subdiagram), the semisimple element $\tilde h$ of $\tilde\g$ is determined by this grading, while $\tilde f$, which we know to exist by \ref{coreduction}, is uniquely determined by $\tilde e$ and $\tilde h$. Thus having an explicit construction of $\tilde f$ would provide an alternative general proof for \ref{coreduction} that would not require separate calculations for the exceptional types. One possibility that comes to mind is to produce $\tilde f$ from $f$ in the same way as we produced $\tilde e$ from $e$ in \ref{e1} --- that is, take $\tilde f=f_1$ where $f=f_1+f_2$ is the unique decomposition of $f\in\g_{-2}$ into a sum of $f_1\in[\g_{-1},\g_{-1}]$ and $f_2\in\g_{-2}(2)$, the latter being the $\g_0$-submodule of $\g_{-2}$ generated by the root vectors corresponding to negatives of the simple roots with weights 2 on the initial weighted Dynkin diagram. However as the following example shows, in general this does not give the correct $\tilde f$.
\end{remark}

\begin{example}
For $\g$ of type D$_6$, consider the nilpotent orbit corresponding to the weighted Dynkin diagram $\fordsix{.2}{1}201011$ (and to the partition 5,3,2,2). One of the nilpotents in this orbit is the following sum of positive root vectors
$$
e:=e_{\fordsix{.1}{.5}110000}+e_{\fordsix{.1}{.5}011110}+e_{\fordsix{.1}{.5}001110}+e_{\fordsix{.1}{.5}001101}+e_{\fordsix{.1}{.5}000111}
$$
where the subscripts denote the linear combinations of simple roots that give the corresponding positive roots. The corresponding $f$ in the $\sla_2$-triple for $e$ is the following combination of negative root vectors:
$$
f:=2f_{\fordsix{.1}{.5}100000} + 4f_{\fordsix{.1}{.5}110000} + 2f_{\fordsix{.1}{.5}011110} - 2f_{\fordsix{.1}{.5}011101} + 2f_{\fordsix{.1}{.5}001110} + 4f_{\fordsix{.1}{.5}001101} + f_{\fordsix{.1}{.5}000111},
$$
with subscripts now designating linear combinations of negatives of simple roots. Thus $h=[e,f]$ determines the grading corresponding to the above weighted Dynkin diagram. It is straightforward to check that in the degree 2 subspace $\g_2$, root vectors corresponding to the combinations $\fordsix{.15}{.75}100000$ and $\fordsix{.15}{.75}110000$ of simple roots span the $\g_0$-submodule $\g_2(2)\subseteq\g_2$ generated by the root vector of $\fordsix{.15}{.75}100000$, i.~e. of the simple root with weight 2, while the remaining positive root vectors from $\g_2$ lie in $[\g_1,\g_1]$. Thus according to \ref{e1}, a strictly odd nilpotent $\tilde e=e_1$ in the diagram subalgebra $\tilde\g$ of type D$_5$ corresponding to the subdiagram obtained by omitting the node with weight 2, is obtained by omitting in the sum for $e$ the leftmost summand (the one that lies in $\g_2(2)$). Thus
$$
\tilde e=e_{\fordsix{.1}{.5}011110}+e_{\fordsix{.1}{.5}001110}+e_{\fordsix{.1}{.5}001101}+e_{\fordsix{.1}{.5}000111}.
$$
Now if we attempt to choose for the companion of $\tilde e$ in the $\sla_2$-triple the element $f_1$ obtained in the same way from $f$, i.~e. by omitting in the sum for $f$ the summands that lie in $\g_{-2}(2)$, we obtain
$$
f_1=2f_{\fordsix{.1}{.5}011110} - 2f_{\fordsix{.1}{.5}011101} + 2f_{\fordsix{.1}{.5}001110} + 4f_{\fordsix{.1}{.5}001101} + f_{\fordsix{.1}{.5}000111}.
$$
However it turns out that $[e_1,f_1]$ is not the semisimple element determining the needed grading of $\tilde\g$. As a matter of fact this element is not semisimple, rather it has form
$$
[e_1,f_1]=h'-e_{\fordsix{.1}{.5}010000}
$$
with $h'$ in the Cartan subalgebra of $\tilde\g$. A correct $\tilde f$ (the one with $[\tilde e,\tilde f]=\tilde h$ an element in the Cartan subalgebra of $\tilde\g$ which gives the correct grading of $\tilde\g$) is
$$
\tilde f=2f{\fordsix{.1}{.5}011110} - 2f{\fordsix{.1}{.5}011101} + 2f{\fordsix{.1}{.5}001101} + f{\fordsix{.1}{.5}000111}
$$
and is thus not obtained from $f$ by projecting it to $[\g_{-1},\g_{-1}]$ or in any other readily apparent way.
\end{example}

Let us add that there are also many examples (even for algebras of type A) when the bracket of the projections $[e_1,f_1]$ of $e$ and $f$ is semisimple but does not induce the required grading on $\tilde\g$.

\section{Maximizing abelian subspaces}

We are interested in abelian subspaces of $\g_1$. First of all, one has the following well-known fact.
\begin{proposition}
Dimension of $\g_1$ is even, and the largest possible dimension of an abelian subspace in $\g_1$ is at most $\frac12\dim\g_1$.
\end{proposition}
\begin{proof}
Let $e$ be an element of the orbit, and choose an $\sla_2$-triple $(e,h,f)$ with $e\in\g_2$, and $h$ inducing the grading. Then one may define a bilinear form on $\g_1$ via
$$
(x,y)_f:=\langle f,[x,y]\rangle,
$$
where $\langle-,-\rangle$ is the Killing form. It is well known that the skew-symmetric form $(-,-)_f$ is nondegenerate (since $\ad f:\g_1\to\g_{-1}$ is an isomorphism), so that dimension of $\g_1$ is indeed even. Moreover any commuting elements of $\g_1$ are orthogonal with respect to this form. Since such a form does not possess isotropic subspaces of more than half dimension of the space, we obtain that there are no abelian subspaces of more than half dimension of $\g_1$.
\end{proof}

\begin{remark}
More generally it is known that a nondegenerate skew-symmetric form exists on the homogeneous part $\g_{2i-1}$ of each odd degree --- see \cite[Proposition 1.2]{Pan}. Thus each $\dim\g_{2i-1}$ is even, too.
\end{remark}

We now consider the abelian subalgebras in $\g_1$, separately for simple algebras of classical types (right now) and for algebras of exceptional types (in Section \ref{secomp}).

Let us thus turn to the simple algebras of classical types. For the type A, it has been proved in \cite{Shoji} that a half-dimensional abelian subspace in $\g_1$ exists for any nilpotent orbit.

The central result of this section is the following characterization in terms of the associated partitions, of those strictly odd nilpotent orbits in types B, C or D which admit an abelian subspace of half the dimension in $\g_1$. We will then deduce the general (not necessarily strictly odd) case using strictly odd reductions as in \ref{strodef}.

\begin{theorem}\label{strictheorem}
Given a strictly odd nilpotent in a simple Lie algebra $\g$ of type $\mathrm B$, $\mathrm C$ or $\mathrm D$, there is an abelian subspace of half dimension in $\g_1$ if and only if the partition corresponding to the nilpotent satisfies one of the following conditions:
\begin{itemize}
\item the largest part $\mu$ of the partition is even and there are no other even parts; moreover if $\g$ is of type $\mathrm B$ then $\mu$ has multiplicity $2$.
\item the largest part $\mu$ of the partition is odd, and either there are no other odd parts, or $\g$ is not of type $\mathrm C$, and the only other parts are $\mu-1$ with multiplicity $2$ and $1$ (with any multiplicity).
\end{itemize}
In other words, abelian subspaces of half dimension in $\g_1$ occur precisely for those strictly odd nilpotents which correspond to partitions of the following kind:

\

\begin{tabular}{rllll}
{\bf type C:}&$\left[1^{2\nu_1}3^{2\nu_3}\cdots(2k-1)^{2\nu_{2k-1}}(2k)^\nu\right]$&$(\nu_{2k-1}\nu\ne0)$,& $\left[2^{\nu_2}4^{\nu_4}\cdots(2k)^{\nu_{2k}}(2k+1)^{2\nu}\right]$&$(\nu_{2k}\nu\ne0)$;\\
{\bf type B or D:}&$\left[2^{2\nu_2}4^{2\nu_4}\cdots(2k)^{2\nu_{2k}}(2k+1)^\nu\right]$&$(\nu_{2k}\nu\ne0)$,&$\left[1^{\nu_1}(2k)^2(2k+1)^\nu\right]$&$(\nu_{2k}\nu\ne0)$;\\
{\bf type B:}&$\left[1^{\nu_1}3^{\nu_3}\cdots(2k-1)^{\nu_{2k-1}}(2k)^2\right]$&$(\nu_{2k-1}\ne0)$,\\
{\bf type D:}&$\left[1^{\nu_1}3^{\nu_3}\cdots(2k-1)^{\nu_{2k-1}}(2k)^{2\nu}\right]$&$(\nu_{2k-1}\nu\ne0)$.
\end{tabular}
\end{theorem}

\begin{proof}
It will be convenient to introduce the following notations: for a partition as above, let $m_k$ be the multiplicity of the number $k$ in it. Moreover let $S_k$ be the $h$-eigensubspace with eigenvalue $k$ in the standard representation, and let $s_k$ denote dimension of this subspace, i.~e. multiplicity of the eigenvalue $k$ for $h$.

As recalled in Section 1 above, the adjoint representation can be identified with the symmetric square of the standard one for type C, and with its exterior square for types B and D.

Because of this, clearly the degree 1 part of the adjoint representation is the direct sum of spaces of the form $S_k^*\otimes S_l$ with $l-k=1$, $k\ge0$, and
$$
\dim\g_1=s_0s_1+s_1s_2+...
$$

Now, from the correspondence described in Section \ref{recoll}, one has
\begin{equation}\label{ses}
\begin{aligned}
s_0&=m_1+m_3+m_5+...\\
s_1&=m_2+m_4+m_6+...\\
s_2&=m_3+m_5+m_7+...\\
s_3&=m_4+m_6+m_8+...\\
...\\
s_{\mu-4}&=m_{\mu-3}+m_{\mu-1}\\
s_{\mu-3}&=m_{\mu-2}+m_\mu\\
s_{\mu-2}&=m_{\mu-1}\\
s_{\mu-1}&=m_\mu
\end{aligned}
\end{equation}

Dimension of the subspace $\g_1$ of grading 1 with respect to the corresponding $\sla_2$-triple is thus given by
$$
s_0s_1+s_1s_2+s_2s_3+s_3s_4+...=\sum_{i,j>0}im_im_{i+2j-1}=m_1m_2+2m_2m_3+m_1m_4+3m_3m_4+2m_2m_5+...
$$

Given an abelian subspace in $\g_1$, using \ref{basisprop} we may assume it has a basis consisting of root vectors. In particular, each of our basis vectors is situated in one of the direct summands $S_k^*\otimes S_{k+1}$.

Note that any elements in $S_{k-1}^*\otimes S_k$ and $S_l^*\otimes S_{l+1}$ commute for $l>k$; whereas when $l=k$, we will obtain a non-commuting pair as soon as our basis contains any elements of the form $x\otimes y\in S_{k-1}^*\otimes S_k$ and $y'\otimes z\in S_k^*\otimes S_{k+1}$ with $y$ and $y'$ mutually dual basis elements. We are thus forced to choose non-intersecting subsets $X_k$, $Y_k$ in the weight vector bases of $S_k$ and include in the basis of the abelian subspace only those $x\otimes y$ which satisfy $x\in X_{k-1}$ and $y\in Y_k$. This does not concern $k=\mu-1$, where $\mu-1$ is the maximal occurring eigenvalue of $h$ ($\mu$, as above, is the largest part of the corresponding partition): in $S_{\mu-1}$ we may choose arbitrary subset of the basis without affecting abelianness; and since we are interested in maximal abelian subspaces, we choose the whole basis of $S_{\mu-1}$.

Moreover any such choice of non-intersecting subsets $X_k$, $Y_k$ of bases of $S_k$ gives indeed an abelian subspace, and we may further assume that $X_k\cup Y_k$ is the whole basis, since otherwise our abelian subspace would not be maximal.

The case $k=0$ is special, and depends on the type considered.

Namely, it may happen that two basis vectors, both from $S_0^*\otimes S_1$, do not commute. Two basis elements of this subspace, being the tensor products of basis vectors corresponding to $\pm\eps_i^{(0)}+\eps_j^{(1)}$ and $\pm\eps_k^{(0)}+\eps_l^{(1)}$ respectively, will commute if and only if the sum $\pm\eps_i^{(0)}+\eps_j^{(1)}\pm\eps_k^{(0)}+\eps_l^{(1)}$ is not a root. This implies that the root vector basis of an abelian subspace in $\g_1$ cannot contain root vectors corresponding to both $\pm\eps_i^{(0)}+\eps_j^{(1)}$ and $\mp\eps_i^{(0)}+\eps_k^{(1)}$ for $j\ne k$ (since the sum of these is the root $\eps_j^{(1)}+\eps_k^{(1)}$).

This is the only restriction on $S_0^*\otimes S_1$ for type D. For type C, there is an additional restriction that an abelian subspace of $\g_1$ cannot contain root vectors corresponding to both $\pm\eps_i^{(0)}+\eps_j^{(1)}$ and $\mp\eps_i^{(0)}+\eps_j^{(1)}$ (since the sum of these is the root $2\eps_j^{(1)}$). For type B, an additional restriction is that an abelian subspace of $\g_1$ cannot contain root vectors corresponding to both $(0+)\eps_j^{(1)}$ and $(0+)\eps_k^{(1)}$ for $j\ne k$ (since the sum of these is the root $\eps_j^{(1)}+\eps_k^{(1)}$).

It follows that to obtain a maximal abelian subspace of $\g_1$, in addition to splitting the weight vector basis of $S_1$ into nonintersecting subsets ($X_1$ and its complement $Y_1$), for any weights $\eps^{(1)}_j$ and $\eps^{(1)}_k$ corresponding to a weight basis vector in $X_1$ we have to pick in $S_0^*\otimes S_1$ the root basis elements corresponding either only to $\eps_i^{(0)}+\eps^{(1)}_j$ and $\eps_i^{(0)}+\eps^{(1)}_k$ or only to $-\eps_i^{(0)}+\eps^{(1)}_j$ and $-\eps_i^{(0)}+\eps^{(1)}_k$ for all possible $i$, but not both. Thus the maximal possible number of basis vectors from $S_0^*\otimes S_1$ which we may include in an abelian subspace of $\g_1$ is either $\left[\frac{s_0}2\right]x_1$ (if we choose either only $\eps_i^{(0)}+\eps^{(1)}_j$ or only $-\eps_i^{(0)}+\eps^{(1)}_j$ for all possible $i$ and $j$) or $s_0$, provided we are not in type C and moreover $X_1$ consists of a single element (corresponding to some $\eps^{(1)}_j$, and we choose root basis vectors corresponding to $\pm\eps_i^{(0)}+\eps^{(1)}_j$ for all possible $i$). In addition, if we are in type B, we may add one more root basis vector $v_0\otimes v_1$ with $v_0$ a weight basis vector with zero weight and $v_1$ some weight basis vector from $X_1$.

Thus for the maximal dimension of the piece of an abelian subspace corresponding to $S_0^*\otimes S_1$ we have the following possibilities:
\begin{center}
\begin{tabular}{c|ccc}
&B&C&D\\
\hline
$x_1=0$&0&0&0\\
$x_1=1$&$s_0$&$\frac{s_0}2$&$s_0$\\
$x_1>1$&$\frac{s_0-1}2x_1+1$&$\frac{s_0}2x_1$&$\frac{s_0}2x_1$
\end{tabular}
\end{center}

This results in the following possibilities for the maximal dimension of an abelian subspace in $\g_1$:
\begin{equation}\label{aposs}
\begin{array}{rl}
\frac{s_0-1}2x_1+1+(s_1-x_1)x_2+(s_2-x_2)x_3+...+(s_{\mu-3}-x_{\mu-3})x_{\mu-2}+(s_{\mu-2}-x_{\mu-2})s_{\mu-1}&\text{(for type B);}\\
&\\
\frac{s_0}2x_1+(s_1-x_1)x_2+(s_2-x_2)x_3+...+(s_{\mu-3}-x_{\mu-3})x_{\mu-2}+(s_{\mu-2}-x_{\mu-2})s_{\mu-1}&\text{(for type C or D);}\\
\\
s_0+(s_1-1)x_2+(s_2-x_2)x_3+...+(s_{\mu-3}-x_{\mu-3})x_{\mu-2}+(s_{\mu-2}-x_{\mu-2})s_{\mu-1}&\text{(for type B or D).}
\end{array}
\end{equation}
where $\mu$ is the largest part of the partition.

We thus want to maximize each of these quantities for $0\le x_k\le s_k$, $k=1,...,\mu-2$.
Note that each of them is linear in all of the $x_k$ separately, hence any possible maxima are attained when every $x_k$ is either $0$ or $s_k$. In fact, more is true:
\begin{lemma}
An abelian subspace of maximal possible dimension in $\g_1$ can be obtained either with $x_{2j-1}=0$, $x_{2j}=s_{2j}$ or with $x_{2j-1}=s_{2j-1}$, $x_{2j}=0$ for all $j$.
\end{lemma}
\begin{proof}
Looking at the subsum
$$
...+(s_{k-2}-x_{k-2})x_{k-1}+(s_{k-1}-x_{k-1})x_k+(s_k-x_k)x_{k+1}+...
$$
determining dimension of the abelian subspace, it is easy to see that each of the following changes:
$$
\begin{array}{llcll}
x_{k-1}=0,&x_k=0&\mapsto&x_{k-1}=0,&x_k=s_k,\\
x_{k-1}=s_{k-1},&x_k=s_k&\mapsto&x_{k-1}=s_{k-1},&x_k=0
\end{array}
$$
does not decrease the dimension of the abelian subspace.

Indeed, these changes do not affect any other summands except those in the above subsum. The first change transforms
$$
...+(s_{k-2}-x_{k-2})0+(s_{k-1}-0)0+(s_k-0)x_{k+1}+...\mapsto...+(s_{k-2}-x_{k-2})0+(s_{k-1}-0)s_k+0x_{k+1}+...,
$$
i.~e. changes the sum by the amount equal to the change from $s_kx_{k+1}$ to $s_{k-1}s_k$. But $x_{k+1}\le s_{k+1}$, and $s_{k+1}\le s_{k-1}$ by \eqref{ses}, so that indeed the sum does not decrease.

Similarly, the second change transforms
$$
...+(s_{k-2}-x_{k-2})s_{k-1}+(s_{k-1}-s_{k-1})s_k+(s_k-s_k)x_{k+1}+...\mapsto...+(s_{k-2}-x_{k-2})s_{k-1}+(s_{k-1}-s_{k-1})0+(s_k-0)x_{k+1}+...,
$$
i.~e. changes the sum by the amount equal to the change from $0$ to $s_kx_{k+1}$, which is obviously a nondecreasing change.

Now using the above changes we may arrive at one of the needed choices. For simplicity, let us encode a given choice of $x$'s by a sequence of zeroes and ones (at the $k$th place of the sequence stands zero if $x_k=0$ and one if $x_k=s_k$). We are allowed to perform ``local transformations'' of the kind $\cdots00\cdots\mapsto\cdots01\cdots$ and $\cdots11\cdots\mapsto\cdots10\cdots$. Using one of these transformations we can always shift the place of the leftmost occurrence of two consecutive identical symbols to the right: say, if this leftmost occurrence is $\cdots11\cdots$ we change it to $\cdots10\cdots$ and if it is $\cdots00\cdots$ we change it to $\cdots01\cdots$, and in the worst case the place of the leftmost occurrence of consecutive identical symbols still shifts to the right by at least one position. Thus if we keep applying the appropriate transformations to the leftmost occurrence of consecutive identical symbols we inevitably arrive either at $10101\cdots$ or at $01010\cdots$.
\end{proof}

Applying this in \eqref{aposs} we obtain that the maximal possible dimension of an abelian subspace in $\g_1$ can only be equal to one of the following six sums:
$$
\begin{array}{r|rl}
\frac{s_0-1}2s_1+1+s_2s_3+s_4s_5+...&s_1s_2+s_3s_4+s_5s_6+...&\text{(for type B)}\\
\frac{s_0}2s_1+s_2s_3+s_4s_5+...&s_1s_2+s_3s_4+s_5s_6+...&\text{(for types C, D)}\\
s_0+s_2s_3+s_4s_5+...&s_0+(s_1-1)s_2+s_3s_4+s_5s_6+...&\text{(for types B, D)}
\end{array}
$$
To find out whether there is an abelian subspace of half the dimension in $\g_1$ is thus equivalent to finding out whether subtracting from the dimension of $\g_1$, i.~e. from $s_0s_1+s_1s_2+...$, one of these sums doubled gives zero, i.~e. whether one of the sums
$$
\begin{array}{lr|lrl}
s_0s_1+s_1s_2+...&-2(\frac{s_0-1}2s_1+1+s_2s_3+s_4s_5+...)
&s_0s_1+s_1s_2+...&-2(s_1s_2+s_3s_4+s_5s_6+...)
&\text{(B)}\\
s_0s_1+s_1s_2+...&-2(\frac{s_0}2s_1+s_2s_3+s_4s_5+...)
&s_0s_1+s_1s_2+...&-2(s_1s_2+s_3s_4+s_5s_6+...)
&\text{(C, D)}\\
s_0s_1+s_1s_2+...&-2(s_0+s_2s_3+s_4s_5+...)
&s_0s_1+s_1s_2+...&-2(s_0+(s_1-1)s_2+s_3s_4+s_5s_6+...)
&\text{(B, D)}
\end{array}
$$
is zero.

Simplifying, we obtain respectively
$$
\begin{array}{rl|rll}
s_1-2+&s_1s_2-s_2s_3+s_3s_4-s_4s_5+s_5s_6-...
&&s_0s_1-s_1s_2+s_2s_3-s_3s_4+s_4s_5-...
&\text{(B)}\\
&s_1s_2-s_2s_3+s_3s_4-s_4s_5+...
&&s_0s_1-s_1s_2+s_2s_3-s_3s_4+...
&\text{(C, D)}\\
-2s_0+s_0s_1+&s_1s_2-s_2s_3+s_3s_4-...
&-2s_0+2s_2+&s_0s_1-s_1s_2+s_2s_3-s_3s_4+...
&\text{(B, D)}
\end{array}
$$
Rewriting this further as
$$
\begin{array}{rl|rll}
s_1-2+&(s_1-s_3)s_2+(s_3-s_5)s_4+(s_5-s_7)s_6+...
&(s_0-s_2)s_1+&(s_2-s_4)s_3+(s_4-s_6)s_5+...
&\text{(B)}\\
&(s_1-s_3)s_2+(s_3-s_5)s_4+(s_5-s_7)s_6+...
&(s_0-s_2)s_1+&(s_2-s_4)s_3+(s_4-s_6)s_5+...
&\text{(C, D)}\\
s_0(s_1-2)+&(s_1-s_3)s_2+(s_3-s_5)s_4+...
&(s_0-s_2)(s_1-2)+&(s_2-s_4)s_3+(s_4-s_6)s_5+...
&\text{(B, D)}
\end{array}
$$
and taking \eqref{ses} into account this can be rewritten as
$$
\begin{array}{rl|rll}
s_1-2+&m_2s_2+m_4s_4+m_6s_6+...
&m_1s_1+&m_3s_3+m_5s_5+...
&\text{(B)}\\
&m_2s_2+m_4s_4+m_6s_6+...
&m_1s_1+&m_3s_3+m_5s_5+...
&\text{(C, D)}\\
s_0(s_1-2)+&m_2s_2+m_4s_4+...
&m_1(s_1-2)+&m_3s_3+m_5s_5+...
&\text{(B, D)}
\end{array}
$$

Let us now assume that our nilpotent is strictly odd, which in terms of the corresponding partition means that $m_{\mu-1}>0$ (here as before $\mu$ is the largest nonzero part of the partition). This then implies that all multiplicities $s_i$ are nonzero. Thus to obtain an abelian subspace of half the dimension in $\g_1$ we have the following possibilities:
$$
\begin{array}{rl|rll}
\text{$s_1=2$ and }&\text{$m_{2k}=0$ for $2k<\mu$}
&&\text{$m_{2k-1}=0$ for $2k-1<\mu$}
&\text{(B)}\\
&\text{$m_{2k}=0$ for $2k<\mu$}
&&\text{$m_{2k-1}=0$ for $2k-1<\mu$}
&\text{(C, D)}\\
\text{$s_1=2$ and }&\text{$m_{2k}=0$ for $2k<\mu$}
&\text{$m_1=0$ or $s_1=2$, and }&\text{$m_{2k-1}=0$ for $1<2k-1<\mu$}
&\text{(B, D)}
\end{array}
$$

We now make the following observations, according to the parity of $\mu$:
\begin{itemize}
\item if $\mu$ is odd, then the cases in the first column are not realizable, since they require that the partition has no even parts, while by strict oddity both $m_{\mu-1}$ and $m_\mu$ must be nonzero;
\item if $\mu$ is even, the cases in the second column are not realizable by exactly the same reason.
\end{itemize}

Taking this into account, we are left with the following cases: for $\mu$ even,
$$
\begin{array}{l|ll}
\text{$m_2=m_4=...=m_{\mu-2}=0$, $m_{\mu-1}>0$, $m_\mu=2$}
&\text{---}
&\text{(B)}\\
\text{$m_2=m_4=...=m_{\mu-2}=0$, $m_{\mu-1}>0$, $m_\mu>0$}
&\text{---}
&\text{(C, D)}\\
\text{$m_2=m_4=...=m_{\mu-2}=0$, $m_{\mu-1}>0$, $m_\mu=2$}
&\text{---}
&\text{(B, D)}
\end{array}
$$
and for $\mu$ odd,
$$
\begin{array}{l|ll}
\text{---}
&\text{$m_1=m_3=...=m_{\mu-2}=0$, $m_{\mu-1}>0$, $m_\mu>0$}
&\text{(B)}\\
\text{---}
&\text{$m_1=m_3=...=m_{\mu-2}=0$, $m_{\mu-1}>0$, $m_\mu>0$}
&\text{(C, D)}\\
\text{---}
&\text{$m_3=m_5=...=m_{\mu-2}=0$, $m_\mu>0$ and either $m_1=0$ or $m_2=m_4=...=m_{\mu-3}=0$ and $m_{\mu-1}=2$}
&\text{(B, D)}
\end{array}
$$

Let us also observe the following:
\begin{itemize}
\item for $\mu$ even, the first case is subsumed by the third one;
\item for $\mu$ even, the third case is subsumed by the second one for type D;
\item for $\mu$ odd, the subcase $m_1=0$ of the third case is subsumed by the first one for type B, and by the second one for type D.
\end{itemize}

Taking all of the above into account gives the partitions as described.
\end{proof}

\begin{remark}
One can formulate the theorem more understandably as follows: in case of type C, there is exactly one parity change along the partition, while in cases B or D there might be either one or two parity changes, but if there are two parity changes then there must be only parts equal to $1$, $\mu-1$, $\mu$ and moreover $\mu-1$ must have multiplicity $2$. Moreover for type B there is one more restriction when there is only one parity change: namely, if the largest part is even, its multiplicity must be $2$.
\end{remark}

We now turn to the not necessarily strictly odd nilpotent orbits, using strictly odd reduction from \ref{strodef}. For classical types, its reformulation in terms of partitions is as follows.

\begin{lemma}\label{mainlemma}
Let $\g$ be a simple Lie algebra of classical type, and let $e$ be a nilpotent element of $\g$ corresponding to the partition $[...k^{m_k}\ell^{m_\ell}...n^{m_n}]$, with $...<k<\ell<...<n$ such that $k$ and $\ell$ are of opposite parity while all the larger parts $j$ (those with $\ell\le j\le n$) are of the same parity.

Then the partition $[...k^{m_k}(k+1)^{m_\ell+...+m_n}]$ defines a strictly odd nilpotent in a Lie algebra of the same type, and corresponds to the strictly odd reduction of $e$, as defined in \ref{strodef}.
\end{lemma}
\begin{proof}
Let us begin by noting that the modified partition is indeed suitable for the same type: if this requires that all parts of the same parity as $k$ have even multiplicity, then we have not touched them; while if this requires that all parts of the same parity as $k+1$ are even, then $\ell$ and all larger parts are of the same parity as $k+1$, so each of the multiplicities $m_\ell$, ..., $m_n$ was even, hence their sum is even too, and we indeed stay with the same type. Moreover the corresponding nilpotent is strictly odd since its largest parts are $k$ and $k+1$.

Now following the correspondence between partitions and weighted Dynkin diagrams described above it is easy to see that passing from the original partition to the one modified as described corresponds to the following modification of the weighted Dynkin diagram: one removes all nodes (and weights) from the left until no more $2$'s are left; for types B, C, D that's all, while for type A one has to do removals symmetrically to that from the right end too.

But this precisely means to leave the connected component of the weighted Dynkin diagram that contains nonzero weights, as described in \ref{reduction} above, so that we indeed obtain the strictly odd reduction of $e$.
\end{proof}

\begin{corollary}\label{maincorollary}
Given a nilpotent in a simple Lie algebra $\g$ of classical type $\mathrm B$, $\mathrm C$ or $\mathrm D$, there is an abelian subspace of half dimension in $\g_1$ if and only if the partition corresponding to the nilpotent satisfies the following conditions:

\begin{tabular}{rl}
type $\mathrm C$:&there is no more than one parity change along the partition;\\
types $\mathrm B$ and $\mathrm D$:&there are no more than two parity changes and, if there is at least one parity change then
\end{tabular}
\begin{itemize}
\item if the largest part of the partition is even, then there is only one parity change, and in the $\mathrm B$ case moreover it must be the unique even part and must have multiplicity $2$;
\item if there are two parity changes, then the largest part of the partition is odd, there is a unique even part, it has multiplicity $2$, and all smaller parts are equal to $1$.
\end{itemize}
Thus, abelian subspaces of half dimension in $\g_1$ occur precisely for nilpotents corresponding to partitions of one of the following kind (with $k\le\ell$ throughout):

\begin{tabular}{rl}
{\bf any type:}&$\left[\cdots(2k-2)^{\nu_{2k-2}}(2k)^{\nu_{2k}}(2\ell+1)^{\nu_{2\ell+1}}(2\ell+3)^{\nu_{2\ell+3}}\cdots\right]$;\\
{\bf type C or D:}&$\left[\cdots(2k-3)^{\nu_{2k-3}}(2k-1)^{\nu_{2k-1}}(2\ell)^{\nu_{2\ell}}(2\ell+2)^{\nu_{2\ell+2}}\cdots\right]$;\\
{\bf type B or D:}&$\left[1^{\nu_1}(2k)^2(2\ell+1)^{\nu_{2\ell+1}}(2\ell+3)^{\nu_{2\ell+3}}\cdots\right]$;\\
{\bf type B:}&$\left[\cdots(2k-3)^{\nu_{2k-3}}(2k-1)^{\nu_{2k-1}}(2\ell)^2\right]$,
\end{tabular}
\end{corollary}
\begin{proof}
This follows from \ref{mainlemma}. Indeed the latter shows that $\g_1(e)$ for a nilpotent $e$ corresponding to some partition has an abelian subspace of half dimension if and only if $\tilde\g_1(\tilde e)$, as described in \ref{coreduction}, has such a subspace; and this happens if and only if the partition modified as in \ref{mainlemma} satisfies conditions of \ref{strictheorem}.

It remains to note that a partition is of the kind indicated if and only if the partition obtained from it as in \ref{mainlemma} satisfies conditions of \ref{strictheorem}.
\end{proof}

\section{Computations}\label{secomp}

It thus remains to find out which of the strictly odd nilpotent orbits in simple Lie algebras of exceptional type do possess an abelian subspace of half dimension in degree 1.

For that, we used the computer algebra system {\tt GAP}. In the package {\tt SLA} by Willem A. de Graaf included in this system one can compute with nilpotent orbits of arbitrary semisimple Lie algebras. In particular, one obtains canonical bases consisting of root vectors for the homogeneous subspaces of all degrees in the grading of the Lie algebra induced by a nilpotent element.

Using \ref{basisprop}, we can determine abelian subspaces in $\g_1$ as follows. Let $B$ be the basis of $\g_1$ made from positive root vectors.
Let us construct a graph with the set of vertices $B$, where two vertices $e_\alpha$ and $e_\beta$ are connected with an edge if and only if they do not commute,
that is, if and only if $\alpha+\beta$ is a root. Then by \ref{basisprop}, $\g_1$ possesses an abelian subspace of dimension $d$ if and only if the basis
consisting of root vectors has a subset of cardinality $d$ consisting of pairwise commuting root vectors.

Clearly this is equivalent to the corresponding graph having an \emph{independent set} of cardinality $d$ --- that is, a subset consisting of $d$ vertices such that no two of these vertices are connected by an edge. Hence describing all possible dimensions of abelian subspaces in $\g_1$ reduces to listing all possible cardinalities of
independent subsets in the corresponding graph.

There is another package {\tt GRAPE} by Leonard H. Soicher in {\tt GAP} which can be used to list all independent sets in a finite graph. Using this package we determine independent sets of maximal possible cardinality in the graph corresponding to the nilpotent orbit.

The results are given in the tables below. A {\tt GAP} code for computing maximal dimensions of abelian subspaces in $\g_1$ for arbitrary semisimple Lie algebras is available at \cite{progaddress}. In fact the program can list all subsets of any given cardinality of pairwise commuting elements in the root vector basis.

\

As an illustration, here are two cases for E$_6$.

\begin{examples}
The nilpotent orbit with the weighted Dynkin diagram \EDynkin{1, 1, 0, 0, 0, 1}{1} has $\g_1$ of dimension $14$.
The corresponding graph with 14 vertices and edges connecting vertices corresponding to non-commuting root vectors in $\g_1$ looks as follows:

\vfill

\begin{center}
\includegraphics[scale=.4]{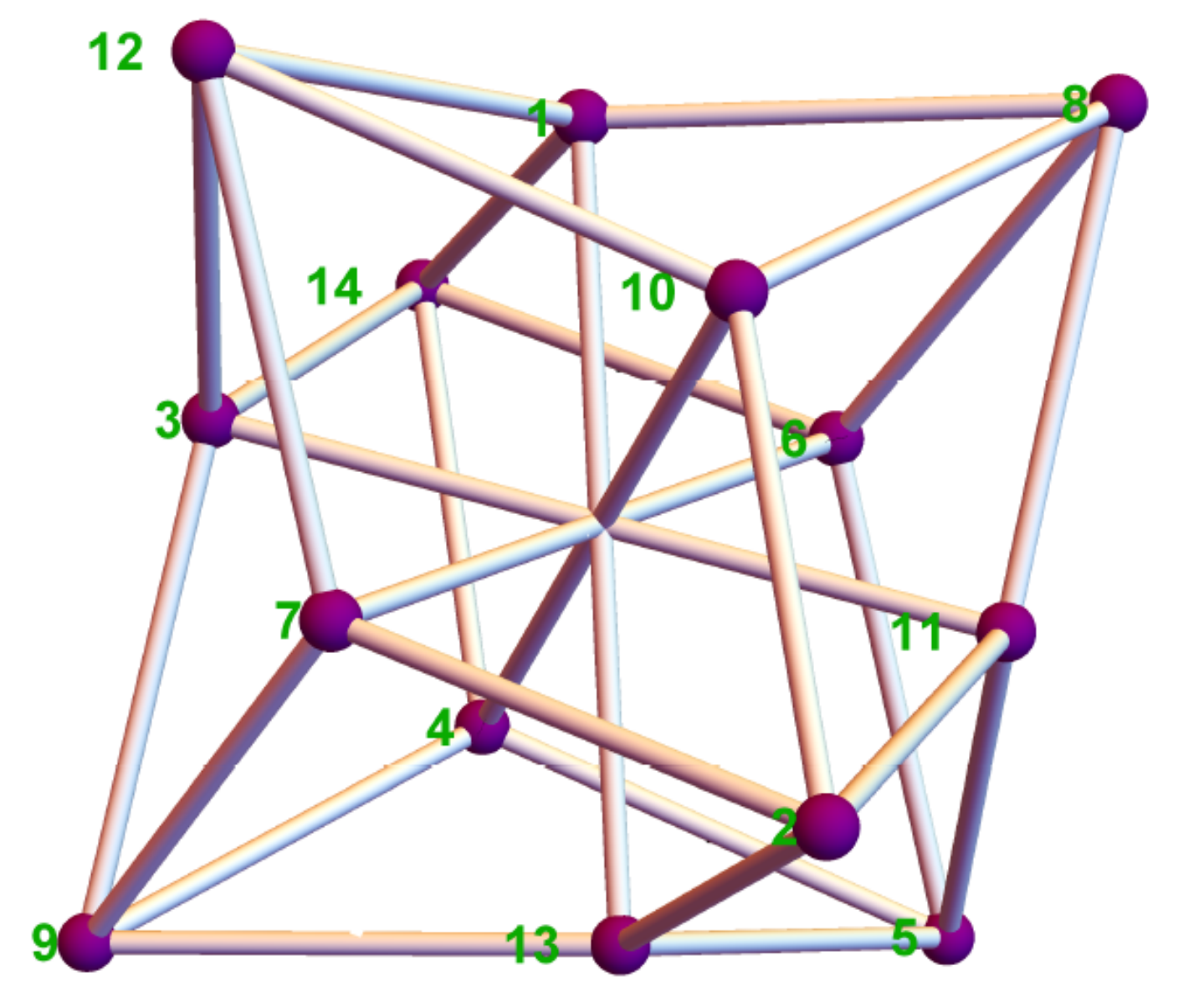}
\end{center}

\vfill

This graph has independent sets with $6$ vertices, e.~g. $\{2,5,8,9,12,14\}$, but any subset on more than $6$ vertices contains a pair of vertices connected with an edge, thus for this nilpotent orbit maximal dimension of an abelian subspace is equal to $6$.

\

Another orbit in E$_6$, with the diagram \EDynkin{1, 0, 1, 0, 1, 0}{1}, has $\g_1$ of dimension $10$ corresponding to the graph

\vfill

\begin{center}
\includegraphics[scale=.4]{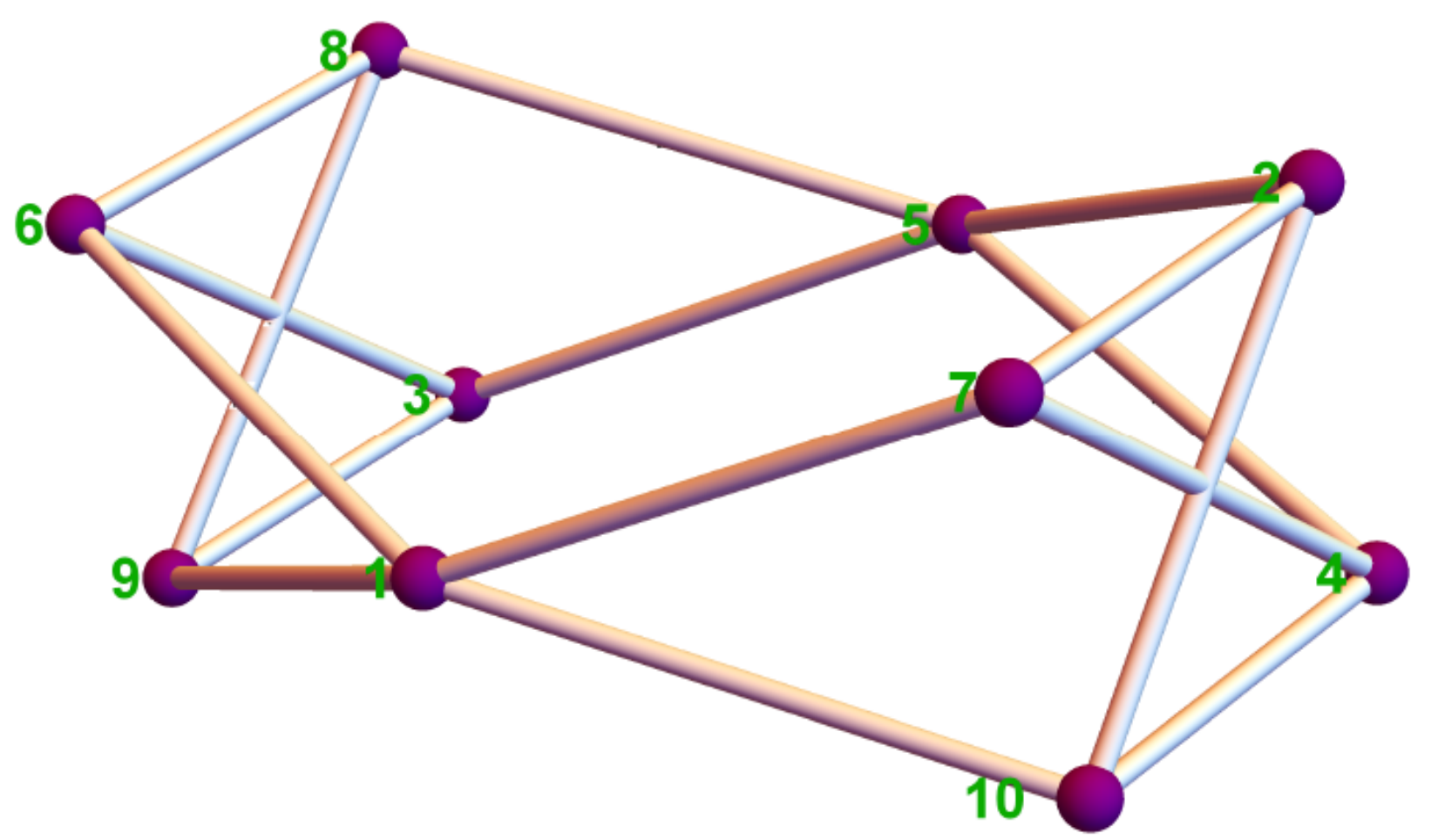}
\end{center}

\vfill

\noindent with 10 vertices. It is easy to find in this graph an independent subset with five elements -- e.~g. $\{1,2,3,4,8\}$.

Thus the orbit in the first example does not possess an abelian subspace of half dimension in $\g_1$, while that in the second one does.
\end{examples}

\section{Tables}

\begin{center}

\def\arraystretch{2.2}
\newlength{\mypar}\setlength{\mypar}{7.67em}

\

\vfill

\begin{tabular}[t]{l|c|c}
\multicolumn{3}{c}{\bf Table G2s}\\
\multicolumn{3}{c}{\parbox[c]{16em}{Strictly odd nilpotent orbits in G$_2$, all with half-abelian $\g_1$:}}\\[2ex]
\hline
Name&Diagram&$\dim\g_1$\\
\hline
&&\\[-5.5ex]
A$_1$&\GDynkin{0, 1}&4\\
$\widetilde{\mathrm A_1}$&\GDynkin{1, 0}&2\\
\end{tabular}

\vfill

\begin{tabular}[t]{l|c|c||l|c|c}
\multicolumn{6}{c}{\bf Table F4s}\\
\multicolumn{6}{c}{Strictly odd nilpotent orbits in F$_4$}\\
\multicolumn{3}{c||}{with half-abelian $\g_1$:}&\multicolumn{3}{c}{without half-abelian $\g_1$:}\\
\hline
Name&Diagram&$\dim\g_1$&Name&Diagram&\parbox[c][4em]{\mypar}{\raggedright$\dim\g_1$ (largest dimension of an abelian subspace)}\\
\hline
&&&&\\[-5.5ex]
A$_1$&\FDynkins{0, 1, 0, 0}&14&$\widetilde{\mathrm A_1}$&\FDynkins{1, 0, 0, 0}&8 (2)\\
A$_1$ $+$ $\widetilde{\mathrm A_1}$&\FDynkins{0, 0, 0, 1}&12&A$_2$ $+$ $\widetilde{\mathrm A_1}$&\FDynkins{0, 0, 1, 0}&6 (2)\\
C$_3(a_1)$&\FDynkins{0, 1, 1, 0}&6&$\widetilde{\mathrm A_2}$ $+$ A$_1$&\FDynkins{1, 0, 0, 1}&8 (3)
\end{tabular}

\vfill

\begin{tabular}{l|c|c||l|c|c}
\multicolumn{6}{c}{\bf Table E6s}\\
\multicolumn{6}{c}{Strictly odd nilpotent orbits in E$_6$}\\
\multicolumn{3}{c||}{with half-abelian $\g_1$:}&\multicolumn{3}{c}{without half-abelian $\g_1$:}\\
\hline
Name&Diagram&$\dim\g_1$&Name&Diagram&\parbox[c][4em]{\mypar}{\raggedright$\dim\g_1$ (largest dimension of an abelian subspace)}\\
\hline
&&&&\\[-5.5ex]
A$_1$&\EDynkin{1, 0, 0, 0, 0, 0}{1}&20&A$_2$ $+$ A$_1$&\EDynkin{1, 1, 0, 0, 0, 1}{1}&14 (6)\\
2A$_1$&\EDynkin{0, 1, 0, 0, 0, 1}{1}&16&2A$_2$ $+$ A$_1$&\EDynkin{0, 1, 0, 1, 0, 1}{1}&12 (5)\\
3A$_1$&\EDynkin{0, 0, 0, 1, 0, 0}{1}&18\\
A$_2$ $+$ 2A$_1$&\EDynkin{0, 0, 1, 0, 1, 0}{1}&12\\
A$_3$ $+$ A$_1$&\EDynkin{1, 0, 1, 0, 1, 0}{1}&10\\
A$_4$ $+$ A$_1$&\EDynkin{1, 1, 1, 0, 1, 1}{1}&8
\end{tabular}

\vfill

\

\pagebreak

\

\vfill

\begin{tabular}{l|c|c||l|c|c}
\multicolumn{6}{c}{\bf Table E7s}\\
\multicolumn{6}{c}{Strictly odd nilpotent orbits in E$_7$}\\
\multicolumn{3}{c||}{with half-abelian $\g_1$:}&\multicolumn{3}{c}{without half-abelian $\g_1$:}\\
\hline
Name&Diagram&$\dim\g_1$&Name&Diagram&\parbox[c][4em]{\mypar}{\raggedright$\dim\g_1$ (largest dimension of an abelian subspace)}\\
\hline
&&&&&\\[-5.5ex]
A$_1$&\EDynkin{0, 1, 0, 0, 0, 0, 0}{1}&32&4A$_1$&\EDynkin{1, 0, 0, 0, 0, 0, 1}{1}&26 (11)\\
2A$_1$&\EDynkin{0, 0, 0, 0, 0, 1, 0}{1}&32&A$_2$ $+$ A$_1$&\EDynkin{0, 1, 0, 0, 0, 1, 0}{1}&24 (9)\\
3A$_1'$&\EDynkin{0, 0, 1, 0, 0, 0, 0}{1}&30&2A$_2$ $+$ A$_1$&\EDynkin{0, 0, 1, 0, 0, 1, 0}{1}&20 (8)\\
A$_2$ $+$ 2A$_1$&\EDynkin{0, 0, 0, 1, 0, 0, 0}{1}&24&A$_3$ $+$ 2A$_1$&\EDynkin{0, 1, 0, 0, 1, 0, 1}{1}&18 (7)\\
(A$_3$ $+$ A$_1$)$'$&\EDynkin{0, 1, 0, 1, 0, 0, 0}{1}&18&A$_4$ $+$ A$_1$&\EDynkin{0, 1, 0, 1, 0, 1, 0}{1}&14 (6)\\
D$_4(a_1)$ $+$ A$_1$&\EDynkin{1, 0, 1, 0, 0, 0, 1}{1}&16\\
A$_3$ $+$ A$_2$&\EDynkin{0, 0, 0, 1, 0, 1, 0}{1}&16
\end{tabular}

\vfill

\

\pagebreak

\

\vfill

\begin{tabular}{l|c|c||l|c|c}
\multicolumn{6}{c}{\bf Table E8s}\\
\multicolumn{6}{c}{Strictly odd nilpotent orbits in E$_8$}\\
\multicolumn{3}{c||}{with half-abelian $\g_1$:}&\multicolumn{3}{c}{without half-abelian $\g_1$:}\\
\hline
Name&Diagram&$\dim\g_1$&Name&Diagram&\parbox[c][4em]{\mypar}{\raggedright$\dim\g_1$ (largest dimension of an abelian subspace)}\\
\hline
&&&&&\\[-5.5ex]
A$_1$&\EDynkin{0, 0, 0, 0, 0, 0, 0, 1}{1}&56&2A$_1$&\EDynkin{0, 1, 0, 0, 0, 0, 0, 0}{1}&64 (22)\\
3A$_1$&\EDynkin{0, 0, 0, 0, 0, 0, 1, 0}{1}&54&4A$_1$&\EDynkin{1, 0, 0, 0, 0, 0, 0, 0}{1}&56 (21)\\
A$_2$ $+$ 3A$_1$&\EDynkin{0, 0, 1, 0, 0, 0, 0, 0}{1}&42&A$_2$ $+$ 2A$_1$&\EDynkin{0, 0, 0, 0, 0, 1, 0, 0}{1}&48 (16)\\
A$_3$ $+$ A$_1$&\EDynkin{0, 0, 0, 0, 0, 1, 0, 1}{1}&34&A$_2$ $+$ A$_1$&\EDynkin{0, 1, 0, 0, 0, 0, 0, 1}{1}&44 (17)\\
A$_3$ $+$ A$_2$ $+$ A$_1$&\EDynkin{0, 0, 0, 1, 0, 0, 0, 0}{1}&30&2A$_2$ $+$ 2A$_1$&\EDynkin{0, 0, 0, 0, 1, 0, 0, 0}{1}&40 (16)\\
A$_4$ $+$ A$_2$ $+$ A$_1$&\EDynkin{0, 0, 1, 0, 0, 1, 0, 0}{1}&24&2A$_2$ $+$ A$_1$&\EDynkin{0, 1, 0, 0, 0, 0, 1, 0}{1}&36 (16)\\
E$_7(a_5)$&\EDynkin{0, 0, 0, 1, 0, 1, 0, 0}{1}&18&A$_3$ $+$ 2A$_1$&\EDynkin{0, 0, 1, 0, 0, 0, 0, 1}{1}&36 (15)\\
A$_6$ $+$ A$_1$&\EDynkin{0, 1, 0, 1, 0, 1, 0, 0}{1}&16&A$_3$ $+$ A$_2$&\EDynkin{0, 1, 0, 0, 0, 1, 0, 0}{1}&32 (13)\\
A$_7$&\EDynkin{0, 1, 0, 1, 0, 1, 1, 0}{1}&14&D$_4(a_1)$ $+$ A$_1$&\EDynkin{1, 0, 0, 0, 0, 0, 1, 0}{1}&32 (12)\\
\multicolumn{3}{c||}{}&2A$_3$&\EDynkin{0, 1, 0, 0, 1, 0, 0, 0}{1}&28 (13)\\
\multicolumn{3}{c||}{}&A$_4$ $+$ 2A$_1$&\EDynkin{0, 0, 0, 1, 0, 0, 0, 1}{1}&28 (12)\\
\multicolumn{3}{c||}{}&A$_4$ $+$ A$_1$&\EDynkin{0, 1, 0, 0, 0, 1, 0, 1}{1}&26 (10)\\
\multicolumn{3}{c||}{}&A$_4$ $+$ A$_3$&\EDynkin{0, 0, 0, 1, 0, 0, 1, 0}{1}&24 (10)\\
\multicolumn{3}{c||}{}&A$_5$ $+$ A$_1$&\EDynkin{0, 1, 0, 1, 0, 0, 0, 1}{1}&22 (9)\\
\multicolumn{3}{c||}{}&D$_5(a_1)$ $+$ A$_2$&\EDynkin{0, 0, 1, 0, 0, 1, 0, 1}{1}&22 (8)\\
\multicolumn{3}{c||}{}&D$_6(a_2)$&\EDynkin{1, 0, 1, 0, 0, 0, 1, 0}{1}&20 (9)\\
\multicolumn{3}{c||}{}&E$_6(a_3)$ $+$ A$_1$&\EDynkin{0, 1, 0, 0, 1, 0, 1, 0}{1}&20 (8)\\
\multicolumn{3}{c||}{}&D$_7(a_2)$&\EDynkin{0, 1, 0, 1, 0, 1, 0, 1}{1}&16 (7)
\end{tabular}

\vfill

\

\pagebreak

\

\vfill

\begin{tabular}[t]{l|c|l}
\multicolumn{3}{c}{\bf Table F4o}\\
\multicolumn{3}{c}{\parbox[c]{18em}{(Non-strictly) odd nilpotent orbits in F$_4$, all with half-abelian $\g_1$:}}\\[2ex]
\hline
Name&Diagram&Strictly odd piece\\
\hline&&\\[-5.5ex]
B$_2$&\FDynkino{1, 2, 0, 0}{3,1}&C$_3$ ($2,1^4$)\\
C$_3$&\FDynkino{2, 1, 1, 0}{4,2}&B$_3$ ($3,2^2$)
\end{tabular}

\

\vfill

\

\begin{tabular}[t]{l|c|l}
\multicolumn{3}{c}{\bf Table E6o}\\
\multicolumn{3}{c}{\parbox[c]{18em}{(Non-strictly) odd nilpotent orbits in E$_6$, all with half-abelian $\g_1$:}}\\[2ex]
\hline
Name&Diagram&Strictly odd piece\\
\hline&&\\[-5.5ex]
A$_3$&\EDynkin{2,1,0,0,0,1}{6,2}&A$_5$\\
A$_5$&\EDynkin{1,2,1,0,1,2}{5,1,3}&D$_4$ ($3,2^2,1$)\\
D$_5(a_1)$&\EDynkin{2,1,1,0,1,1}{6,2}&A$_5$
\end{tabular}

\

\vfill

\

\vfill

\

\begin{tabular}{l|c|l||l|c|l}
\multicolumn{6}{c}{\bf Table E7o}\\
\multicolumn{6}{c}{(Non-strictly) odd nilpotent orbits in E$_7$}\\
\multicolumn{3}{c||}{with half-abelian $\g_1$:}&\multicolumn{3}{c}{without half-abelian $\g_1$:}\\
\hline
Name&Diagram&Strictly odd piece&Name&Diagram&Strictly odd piece\\
\hline&&&&\\[-5.5ex]
A$_3$&\EDynkin{0,2,0,0,0,1,0}{7,1,3}&D$_6$ ($2^2,1^8$)&D$_4$ $+$ A$_1$&\EDynkin{1, 2, 1, 0, 0, 0, 1}{7,1,3}&D$_6$ ($3,2^4,1$)\\
D$_5(a_1)$&\EDynkin{0,2,0,1,0,1,0}{7,1,3}&D$_6$ ($3^2,2^2,1^2$)&A$_5$ $+$ A$_1$&\EDynkin{0, 1, 0, 1, 0, 1, 2}{6,1,2}&E$_6$ (2A$_2$ $+$ A$_1$)\\
A$_5'$&\EDynkin{0,1,0,1,0,2,0}{5,1,2}&D$_5$ ($3,2^2,1^3$)\\
D$_6(a_2)$&\EDynkin{1,0,1,0,1,0,2}{6,1,2}&E$_6$ (A$_3$ $+$ A$_1$)\\
D$_5$ $+$ A$_1$&\EDynkin{1,2,1,0,1,1,0}{7,1,3}&D$_6$ ($4^2,3,1$)\\
D$_6(a_1)$&\EDynkin{1,2,1,0,1,0,2}{6,1,3}&D$_5$ ($3^2,2^2$)\\
D$_6$&\EDynkin{1,2,1,0,1,2,2}{5,1,3}&D$_4$ ($3,2^2,1$)
\end{tabular}

\

\vfill

\

\vfill

\

\begin{tabular}{l|c|l||l|c|l}
\multicolumn{6}{c}{\bf Table E8o}\\
\multicolumn{6}{c}{(Non-strictly) odd nilpotent orbits in E$_8$}\\
\multicolumn{3}{c||}{with half-abelian $\g_1$:}&\multicolumn{3}{c}{without half-abelian $\g_1$:}\\
\hline
Name&Diagram&Strictly odd piece&Name&Diagram&Strictly odd piece\\
\hline&&&&&\\[-5.5ex]
A$_3$&\EDynkin{0,1,0,0,0,0,0,2}{7,1,2}&E$_7$ (A$_1$)&D$_4$ $+$ A$_1$&\EDynkin{1, 0, 0, 0, 0, 0, 1, 2}{7,1,2}&E$_7$ (4A$_1$)\\
D$_5(a_1)$ $+$ A$_1$&\EDynkin{0,0,0,1,0,0,0,2}{7,1,2}&E$_7$ (A$_2$ $+$ 2A$_1$)&D$_5(a_1)$&\EDynkin{0, 1, 0, 0, 0, 1, 0, 2}{7,1,2}&E$_7$ (A$_2$ $+$ A$_1$)\\
A$_5$&\EDynkin{0,2,0,0,0,1,0,1}{8,1,3}&D$_7$ ($3,2^2,1^7$)&D$_5$ $+$ A$_1$&\EDynkin{0, 1, 0, 0, 1, 0, 1, 2}{7,1,2}&E$_7$ (A$_3$ $+$ 2A$_1$)\\
D$_6(a_1)$&\EDynkin{1,0,1,0,0,0,1,2}{7,1,2}&E$_7$ (D$_4(a_1)$ $+$ A$_1$)&E$_6(a_1)$ $+$ A$_1$&\EDynkin{0, 1, 0, 1, 0, 1, 0, 2}{7,1,2}&E$_7$ (A$_4$ $+$ A$_1$)\\
E$_7(a_4)$&\EDynkin{0,0,0,1,0,1,0,2}{7,1,2}&E$_7$ (A$_3$ $+$ A$_2$)&D$_6$&\EDynkin{1, 2, 1, 0, 0, 0, 1, 2}{7,1,3}&D$_6$ ($3,2^4,1$)\\
E$_7(a_3)$&\EDynkin{0,2,0,1,0,1,0,2}{7,1,3}&D$_6$ ($3^2,2^2,1^2$)&E$_6$ $+$ A$_1$&\EDynkin{0, 1, 0, 1, 0, 1, 2, 2}{6,1,2}&E$_6$ (2A$_2$ $+$ A$_1$)\\
D$_7$&\EDynkin{1,2,1,0,1,1,0,1}{8,1,3}&D$_7$ ($5,4^2,1$)\\
E$_7(a_2)$&\EDynkin{1,0,1,0,1,0,2,2}{6,1,2}&E$_6$ (A$_3$ $+$ A$_1$)\\
E$_7(a_1)$&\EDynkin{1,2,1,0,1,0,2,2}{6,1,3}&D$_5$ ($3^2,2^2$)\\
E$_7$&\EDynkin{1,2,1,0,1,2,2,2}{5,1,3}&D$_4$ ($3,2^2,1$)
\end{tabular}
\end{center}

\section*{Acknowledgements}
The authors are grateful to the referee for highly professional work, including simplifications of proofs and improvements of exposition.

The third named author wishes to thank Daniele Valeri for discussions on generalization of Miura maps, constructed in \cite{dSKV}.

The second named author gratefully acknowledges help of the user {\tt marmot} from {\tt tex.stackexchange.com} in producing the \TeX\ code that was used to highlight the strictly odd pieces of weighted Dynkin diagrams in the last four tables.

\bibliographystyle{amsplain}

\end{document}